\documentclass{amsart}
\date{15 Mar., 2012}
\usepackage{amsmath,amsthm,amssymb,mathrsfs,enumerate,bm}
\usepackage{comment}
\usepackage{color}

\title[An excess theorem for spherical $2$-designs]
{An excess theorem for spherical $2$-designs}

\author[H.~Kurihara]
{Hirotake Kurihara}
\thanks{The author is supported by JSPS Research Fellowship.}
\address{Mathematical Institute,
Tohoku University,
Aramaki-Aza-Aoba 6-3,
Aoba-ku,
Sendai 980-8578,
Japan
}
\email{sa9d05@math.tohoku.ac.jp}

\numberwithin{equation}{section}
\newtheorem{thm}{Theorem}[section]
\newtheorem*{thm*}{Theorem}
\newtheorem{lem}[thm]{Lemma}

\theoremstyle{definition}
\newtheorem{df}[thm]{Definition}
\theoremstyle{remark}
\newtheorem{rem}[thm]{Remark}

\newcommand{\R}{\mathbb{R}}

\def\set#1#2{\{#1\,|\,#2\}}
\newcommand{\trans}[1]{{}^t\hspace{-0.4ex}#1}

\newcommand{\rank}{\mathrm{rank}\,}
\newcommand{\tr}{\mathrm{tr}\,}
\newcommand{\Sp}{\mathrm{Spec}\,}
\newcommand{\pol}{\mathrm{Pol}}
\newcommand{\Harm}{\mathrm{Harm}}
\newcommand{\Span}{\mathrm{Span}}

\newcommand{\aprod}[2]{
\langle #1,#2 \rangle
}
\newcommand{\mprod}[2]{
\langle #1,#2 \rangle_{\mathrm{m}}
}

\begin{document}
\begin{abstract}
We give an excess theorem for spherical $2$-designs.
This theorem is a dual version
of the spectral excess theorem for graphs,
which gives a characterization of
distance-regular graphs,
among regular graphs in terms of the eigenvalues
and the excess.
Here we give a characterization
of $Q$-polynomial association schemes
among spherical $2$-designs.
\end{abstract}

\maketitle

\section{Introduction}
In graph theory, distance-regularity is considered an
important concept,
and there are many papers on distance-regular graphs.
The reader is referred to
Brouwer--Cohen--Neumaier~\cite{Brouwer1989dg}.
Especially the methods to check whether a given graph is
distance-regular or not are of interest.
One of these methods is
the \emph{spectral excess theorem}~\cite{Fiol1997fla,
Dam2008set,Fiol2010spo}.
The purpose of the present paper is to give a
``dual version'' of this theorem.

We briefly recall the spectral excess theorem.
For the rest of this section,
$\Gamma =(X,E)$ denotes an undirected, simple,
connected and finite graph of order $n$
and diameter $D(\Gamma)=D$.
The path distance between two vertices $x$ and $y$
is represented by $\partial (x,y)$.
Let $\Gamma _i(x)$ be the set of vertices at distance $i$ from $x$,
for $0\le i\le D$.
The $i$-th \emph{distance matrix} $A_i$ of $\Gamma$ is defined to be
the matrix with rows and columns indexed by $X$
whose $(x,y)$-entries are
\[
(A_i)_{x,y}=
\begin{cases}
1 & \text{if $y\in \Gamma _i(x)$,}\\
0 & \text{otherwise}.
\end{cases}
\]
In particular $A_1$ is the \emph{adjacency matrix} of
$\Gamma$ and denoted simply by $A$.
The {\it spectrum} of $\Gamma$ is the set of distinct eigenvalues
of $A$ together with their multiplicities,
and it will be denoted by
$\Sp (\Gamma)
=
\{\theta _0^{m_0},\theta _1^{m_1},\ldots ,\theta _d^{m_d}\}$.
Then we have $D\le d$.
Suppose $\Gamma$ is a connected $k$-regular graph.
Then $k$ is an eigenvalue of $\Gamma$ with multiplicity 1.
Hence we may put
$\theta_0=k$
and
$m_0=1$.
Let $Z(t)=\prod^d_{i=0}(t-\theta_i)\in \R[t]$,
and we consider
an inner product on $\R[t]/(Z)$ by,
for $p,q\in \R[t]/(Z)$,
\[
\aprod{p}{q}
=
\frac{1}{n}
\sum^{d}_{i=0}m_i p(\theta_i)q(\theta_i).
\]
The \emph{predistance polynomials} $p_0,p_1,\ldots ,p_d$ of $\Gamma$
are the unique polynomials
satisfying $\deg p_i=i$
and
$\aprod{p_i}{p_j}=\delta_{i,j} p_i(\theta_0)$
for $i,j \in \{0,1,\ldots ,d\}$, where
$\delta _{i,j}$ denotes Kronecker's delta function.
The graph $\Gamma =(X,E)$ is called
{\it distance-regular}
if $D=d$, and $A_i=p_i(A)$
for each $i\in \{0,1,\ldots ,d\}$.

It is well known that regularity of a graph is determined by
the spectrum of the graph,
cf.~Cvetkovi{\'c}--Doob--Sachs~\cite[Theorem 3.22]{Cvetkovic1980sog}.
On the other hand,
distance-regularity of a graph is in general not determined
only by the spectrum of the graph.
The spectral excess theorem states 
a condition that a connected regular graph
is distance-regular.
Here the excess of a vertex $x$ is $|\Gamma _d(x)|$.
This result was first proved by Fiol--Garriga~\cite{Fiol1997fla}
using a local approach.\ 
Later, Van Dam~\cite{Dam2008set} and
Fiol--Gago--Garriga~\cite{Fiol2010spo}
gave the excess theorem using mean excess,
which used a global approach.
The following is the spectral excess theorem due to
Fiol--Gago--Garriga~\cite{Fiol2010spo}.
\begin{thm}
Let $\Gamma$ be a regular graph with $d+1$
distinct eigenvalues and spectrally maximum diameter
$D=d$.
Then
$\frac{1}{n}\sum_{x\in X}|\Gamma_d(x)| \le p_d(\theta _0)$.
Moreover equality is attained if and only if
$\Gamma$ is a distance-regular graph.
\end{thm}
Every distance-regular graph has the structure of
a \emph{$P$-polynomial association scheme}.
The notation and some properties about association schemes are given
in Section~\ref{sec:AS}.
The \emph{$Q$-polynomial property} is 
an algebraically dual concept of the $P$-polynomial property.

In this paper we consider an excess theorem related to
$Q$-polynomial association schemes.
Here, the ``dual'' concept of regular graphs is
corresponding to properties of
spherical $2$-designs.
In Section~\ref{sec:pol_sp},
the definition of spherical designs is given,
and we consider the polynomial functions on 
spherical 2-designs.
Our main result is given and proved in Section~\ref{sec:excess}.

\section{Association schemes}
\label{sec:AS}
We begin with a review of basic definitions concerning
association schemes.
The reader is referred to Bannai--Ito~\cite{Bannai1984aci} for the
background material.

A {\it symmetric association scheme}
$\mathfrak{X}=(X,\{R_i\}^{d}_{i=0})$
consists of a finite set $X$ and a set $\{R_i\}^{d}_{i=0}$
of non-empty binary relations on $X$ satisfying:
\begin{enumerate}
\item $R_0=\set{(x,x)}{x\in X}$,
\item $\{R_i\}^{d}_{i=0}$ is a partition of $X\times X$,
\item $\trans{R_i}=R_i$ for each $i\in \{0,1,\ldots ,d\}$,
where $\trans{R_i}=\set{(y,x)}{(x,y)\in R_i}$,
\item the numbers $|\set{z\in X}{\text{$(x,z)\in R_i$ and $(z,y)\in
R_j$}}|$ are constant whenever $(x,y)\in R_k$, for each $i,j,k\in
\{0,1,\ldots ,d\}$.
\end{enumerate}
The numbers $|\set{z\in X}{\text{$(x,z)\in R_i$ and
$(z,y) \in R_j$}}|$ are called the {\it intersection numbers}
and denoted by $p^{k}_{i,j}$.
Let $M_X(\R)$ denote the algebra of matrices over
the real field $\R$ with rows and columns indexed by $X$.
The $i$-th {\it adjacency matrix} $A_i$ in $M_X(\R)$ of
$\mathfrak{X}$ is defined by
\[(A_i)_{x,y}=\begin{cases}
1 & \text{if $(x,y)\in R_i$,}\\
0 & \text{otherwise.}
\end{cases}\]
From the definition of association schemes, it follows that
\begin{enumerate}[{(A}1{)}]
\item
$A_0=I$,
where $I$ is the identity matrix,
\item 
$A_0+A_1+\cdots +A_d=J$,
where $J$ is the all-one matrix,
and $A_i \circ A_j=\delta_{i,j}A_i$
for $i,j\in \{0,1,\ldots ,d\}$,
where
$\circ$ denotes the Hadamard product,
that is,
the entry-wise matrix product,
\item 
$\trans{A_i}=A_i$ for each $i\in \{0,1,\ldots ,d\}$,
\item 
$A_i A_j=\sum^{d}_{k=0}p^k_{i,j}A_k$,
for each $i,j\in \{0,1,\ldots ,d\}$.
\end{enumerate}
The vector space
$\mathfrak{A}=\Span _\R \{A_0,A_1,\ldots, A_d\}$
with a basis $\{A_i\}^d_{i=0}$ forms a commutative algebra
and is called the {\it Bose--Mesner algebra} of
$\mathfrak{X}$.
It is well known that $\mathfrak{A}$ is semi-simple,
hence $\mathfrak{A}$ has a second basis
$E_0, E_1,\ldots ,E_d$ satisfying the following conditions:
\begin{enumerate}[{(E}1{)}]
\item 
$E_0=\frac{1}{|X|}J$,
\item 
$E_0+E_1+\cdots +E_d=I$ and $E_i E_j =\delta_{i,j} E_i$,
\item 
$\trans{E_i}=E_i$ for each $i\in \{0,1,\ldots ,d\}$,
\item 
$E_i \circ E_j=\frac{1}{|X|}\sum^{d}_{k=0}q^k_{i,j}E_k$
for some real numbers $q^k_{i,j}$,
for each $i,j\in \{0,1,\ldots ,d\}$.
\end{enumerate}
Then $E_0,E_1,\ldots ,E_d$ are the primitive idempotents
of the Bose--Mesner algebra $\mathfrak{A}$.
The {\it first eigenmatrix} $P=(P_i(j))^d_{j,i=0}$
and the {\it second eigenmatrix} $Q=(Q_i(j))^d_{j,i=0}$
of $\mathfrak{X}$ are defined by 
\[A_i=\sum^{d}_{j=0}P_i(j)E_j\ 
\text{and}\ 
E_i=\frac{1}{|X|}\sum^{d}_{j=0}Q_i(j)A_j,\]
respectively.

We call $\mathfrak{X}$ a \emph{$P$-polynomial scheme}
(or a \emph{metric scheme})
with respect to the ordering $\{A_i \}_{i=0}^d$,
if for each $i \in \{0,1,\ldots, d \}$, 
there exists a polynomial $v_i$ of degree $i$,
such that $A_i=v_i(A_1)$.
Moreover $\mathfrak{X}$ is called
a $P$-polynomial scheme with respect to $A_1$
if it has the $P$-polynomial property
with respect to some ordering
$A_0,A_1, A_{i_2},A_{i_3},\ldots ,A_{i_d}$.
Dually, $\mathfrak{X}$ is called
a \emph{$Q$-polynomial scheme} (or a \emph{cometric scheme}) 
with respect to the ordering $\{E_i \}_{i=0}^d$,
if for each $i \in \{0,1,\ldots, d \}$, 
there exists a polynomial $v_i^{\ast}$ of degree $i$, 
such that $|X| E_i=v^\ast _i((|X| E_1)^\circ)$,
where, for $f\in \R [t]$ and
$M=(M_{x,y})_{x,y\in X}\in M_X(\R)$,
we define $f(M^\circ)=(f(M_{x,y}))_{x,y\in X}$.
Moreover $\mathfrak{X}$ is called
a $Q$-polynomial scheme with respect to $E_1$
if it has the $Q$-polynomial property
with respect to some ordering
$E_0,E_1, E_{i_2},E_{i_3},\ldots ,E_{i_d}$.
In fact, an ordering of a $Q$-polynomial association scheme
with respect to $E_1$ is uniquely determined
(cf.~Kurihara--Nozaki~\cite{Kurihara2012coq}).
It is known that $v_i$ and $v_i^{\ast}$
form systems of orthogonal polynomials~\cite{Bannai1984aci}.

\section{Harmonic polynomials on Spherical $2$-designs
and Predegree polynomials}
\label{sec:pol_sp}
In this section,
we consider the decomposition of the space of functions
on a spherical $2$-design into
the subspaces consisting of harmonic functions of the same degree.
Also we consider the predegree polynomials,
which are determined by the angles among
the elements of the spherical $2$-design.

Let $r \mathbb{S}^{m-1}=\set{(x_1,x_2,\ldots ,x_m)\in \mathbb{R}^m}{x_1^2+x_2^2+\cdots +x_m^2=r^2}$ be the sphere of radius $r$ centered at the origin in $\mathbb{R}^m$, endowed with the standard inner product
$\bm{x}\cdot \bm{y}=\sum^m_{i=1}x_i y_i$
for $\bm{x}=(x_1,x_2,\ldots ,x_m),
\bm{y}=(y_1,y_2,\ldots ,y_m)\in \R^m$.
In particular $1 \mathbb{S}^{m-1}$ is denoted simply by $\mathbb{S}^{m-1}$.
For a nonempty finite set $X$ in $r \mathbb{S}^{m-1}$, we set
$A(X):=\set{\bm{x}\cdot \bm{y}}{\bm{x},\bm{y}\in X,\ \bm{x}\neq \bm{y}}$
and $A'(X):=A(X)\cup \{r^2\}$,
$X$ is called an \emph{$s$-distance set}
if $|A(X)|=s$.

The concept of spherical designs was introduced
by Delsarte--Goethals--Seidel~\cite{Delsarte1977sca},
and we refer also to Bannai--Bannai~\cite{Bannai2009ssd}
and
\cite{Delsarte1977sca} for detail of
spherical designs.
\begin{df}[Spherical design]
Let $\tau$ be a nonnegative integer. 
A finite nonempty subset $X$ of $r \mathbb{S}^{m-1}$
is called a spherical $\tau$-design if
\[
\frac{1}{\mu (r \mathbb{S}^{m-1})} \int_{r \mathbb{S}^{m-1}} f(\bm{x})d\mu (\bm{x})
=
\frac{1}{|X|}\sum_{\bm{x}\in X} f(\bm{x})
\]
holds for all polynomials $f(\bm{x})=f(x_1,x_2,\ldots,x_m)$
of degree at most $\tau$.
Here $\mu $ is the Lebesgue measure on $r\mathbb{S}^{m-1}$.
\end{df}

Let $X$ be a nonempty finite set in $\sqrt{m}\mathbb{S}^{m-1}$
of size $n$.
For $\alpha \in A'(X)$,
the \emph{relation matrix} $A_\alpha$ is
defined to be
the matrix in $M_X(\R)$ with
\[
(A_\alpha)_{\bm{x},\bm{y}}=
\begin{cases}
1 & \text{if $\bm{x}\cdot \bm{y}=\alpha$,}\\
0 & \text{otherwise.}
\end{cases}
\]
The \emph{normalized Gram matrix} $G$ of $X$
is defined by $G:=\frac{1}{n}\sum_{\alpha \in A'(X)} \alpha A_\alpha$.
Since $X$ is in the sphere of radius $\sqrt{m}$,
$X$ satisfies
\begin{enumerate}[{(TD}1{)}]
\item \label{def:poly_sp_3}
$\bm{x}\cdot \bm{x}=m$ for each $\bm{x}\in X$, i.e.,
$G \circ I=\frac{m}{n} I$,
\item \label{def:poly_sp_4}
$\bm{x}\cdot \bm{y}< m$ for each $\bm{x}, \bm{y}\in X$
with $\bm{x}\neq \bm{y}$, i.e.,
$A_m=I$.
\end{enumerate}
Moreover
it is known that $X$ is a spherical $2$-design
if and only if
\begin{enumerate}[{(TD}1{)}]
\setcounter{enumi}{2}
\item \label{def:poly_sp_1}
$\sum_{\bm{y}\in X}\bm{x}\cdot \bm{y}=0$ for each $\bm{x}\in X$, i.e.,
$G J=0$,
\item \label{def:poly_sp_2}
$\frac{1}{n}\sum_{\bm{z}\in X}(\bm{x}\cdot \bm{z})
(\bm{z}\cdot \bm{y})
=\bm{x}\cdot \bm{y}$ for each $\bm{x}, \bm{y}\in X$, i.e.,
$G^2 =G$,
\end{enumerate}
(cf.~\cite{Delsarte1977sca}).
This condition describe that
the center of $X$ is the origin of $\R ^m$
and $X$ is an eutectic star~\cite{Coxeter1973rp,Delsarte1977sca}.

\begin{rem}
The condition (TD1)--(TD4) is an algebraically dual condition of
simple, connected and regular graphs.
In fact, a graph of order $n$ is simple, $k$-regular, connected
if and only if
the adjacency matrix $A$ of the graph satisfies
$A\circ I=0$, $A\circ A=A$,
$A J=k J$ and
$E_0=\frac{1}{n}J$,
where $E_0$ is the
orthogonal projection onto the eigenspace of $A$
with eigenvalue $k$.
This duality is similar to the duality
of the adjacency matrices and
the primitive idempotents of
association schemes.
\end{rem}

For a $Q$-polynomial association scheme
$\mathfrak{X}=(X,\{R_i\}^s_{i=0})$
with respect to $E_1$,
the mapping $x\mapsto (|X|E_1 )_x$,
sending each vertex $x$ to the $x$-th column
of $|X| E_1$,
determines a mapping from $X$ to a finite subset $\tilde{X}$
in $\Span_\R \set{(|X| E_1 )_y}{y\in X}\cong \R^{m}$,
where $m=\rank E_1$.
It is known that
$\tilde{X}$ can be an $s$-distance set in
$\sqrt{m} \mathbb{S}^{m-1}$,
and $\tilde{X}$ is always a spherical 2-design
(cf.~Cameron--Goethals--Seidel~\cite{Cameron1978kcs}).
On the other hand,
not all spherical 2-designs arise in this way.
We are interested in when 
a spherical 2-design
has the structure of a
$Q$-polynomial association scheme.
There exists a result for this ``$Q$-polynomial property''
of spherical designs.
Delsarte--Goethals--Seidel~\cite{Delsarte1977sca}
proved that
a $\tau$-design $X$ which is an $s$-distance set
with $\tau\ge 2s-2$
has the $Q$-polynomial property.
In fact the condition $\tau\ge 2s-2$ seems 
too strong.
In the present paper,
we only focus on the $Q$-polynomial property of
$2$-designs.


For a finite subset $X$
in $\sqrt{m}\mathbb{S}^{m-1}$ of size $n$,
the vector space of $\R$-valued functions on $X$ is denoted by $C(X)$.
We equip $C(X)$ with an inner product by
\[
(f,g)=\frac{1}{n}\sum_{\bm{x}\in X} f(\bm{x}) g(\bm{x})
\]
for $f,g \in C(X)$.
For a polynomial $p\in \R [t]$ and $\bm{a}\in X$,
we define the \emph{zonal polynomial $\zeta _{\bm{a}}(p) :X\rightarrow \R$
of $p$ at $\bm{a}$}
by $\zeta _{\bm{a}}(p) (x)=p(\bm{a}\cdot \bm{x})$.
We further define the vector spaces $\pol _k(X)$,
which first appeared in Godsil~\cite{Godsil1989ps},
recursively by setting:
\begin{itemize}
\item $\pol_0(X)$ is the set of constant functions on $X$,
\item $\pol_1(X)=\Span _\R \set{\zeta _{\bm{a}}(p)}{\bm{a}\in X, \deg p\le 1}$,
\item $\pol_{k}(X)=\Span _\R  \set{f g}
{f\in \pol_1(X),\ g\in \pol_{k-1
}(X)}$ for $k \ge 2$.
\end{itemize}
If $f\in \pol_k(X)\setminus \pol_{k-1}(X)$,
then we say that $f$ has \emph{degree $k$}.
From the definition of $\pol_{k}(X)$,
it seems complicated to describe
elements of $\pol_{k}(X)$,
but
we can simplify the description of $\pol _k (X)$.
\begin{lem}[cf.~Theorem 4.1 in Chapter 15 in Godsil~\cite{Godsil1993ac}]
\label{lem:Pol_span_easy}
For each nonnegative integer $k$,
\[
\pol _k (X)=\Span _\R \set{\zeta _{\bm{a}}(p)}{\bm{a}\in X, \deg p\le k}.
\]
\end{lem}

Next we define the degree of $X$.
\begin{lem}
\label{lem:C(X)eqPol_s}
Suppose a finite set $X \subset \sqrt{m}\mathbb{S}^{m-1}$ is an $s$-distance set.
Then
\[C(X)
=
\pol _s (X).
\]
\end{lem}
\begin{proof}
Let $F(t)=\prod_{\alpha \in A(X)}
\frac{t-\alpha}{m-\alpha}$.
Then
by (TD\ref{def:poly_sp_3}) and
(TD\ref{def:poly_sp_4}),
for each $\bm{y}\in X$, the function $\zeta _{\bm{y}}(F) (\bm{x})$
is equal to 1 if $\bm{x}=\bm{y}$, and 0 otherwise.
Namely $\{\zeta _{\bm{y}}(F)\}_{\bm{y}\in X}$ is a basis
of $C(X)$.
Since $\deg F=s$, we have $\zeta _{\bm{y}}(F) \in \pol _s(X)$
for all $\bm{y}\in X$.
Thus, $C(X)= \pol _s(X)$, as desired.
\end{proof}
We set $S=\min\set{i\in \{0,1,\ldots ,s\}}{\pol _i(X )=C(X)}$,
and we call $S$ the \emph{degree} of $X$.
By Lemma~\ref{lem:C(X)eqPol_s}, we have
$S\le s$.

For the rest of this paper we always suppose that
$X$ is a spherical 2-design of size $n$ in $\sqrt{m}\mathbb{S}^{m-1}$.
Let $\Harm_0(X)=\pol_0(X)$
and we define
\[
\Harm_k (X)=\pol_k(X)
\cap
\pol_{k-1}(X)^\bot
\ \text{for $k\ge 1$}.\]
Its elements are \emph{harmonic polynomials of degree $k$}.
From the definition of the degree of $X$, we get
\[
\Harm _j(X) \neq \{0\}\ (0\le j\le S)\ 
\text{and}
\ 
C(X)=\bigoplus^{S}_{i=0}\Harm_i (X).
\]
By (TD\ref{def:poly_sp_1}),
for a polynomial $r_1 t+r_2 \in \R [t]$ of degree 1 and $\bm{a}\in X$,
we have
\begin{eqnarray*}
(\pmb{1},\zeta_{\bm{a}} (r_1 t+r_2))
&=&
\frac{1}{n}\sum_{\bm{x}\in X}(r_1 \bm{a}\cdot \bm{x}+r_2)\\
&=&
\frac{r_1}{n}
\sum_{\bm{x}\in X}\bm{a}\cdot \bm{x}+r_2\\
&=&
r_2,
\end{eqnarray*}
where
$\pmb{1} (\bm{x})=1$ for any $\bm{x}\in X$.
Hence
$\Harm_1 (X)=\mathrm{Span}_{\R} \set{\zeta_{\bm{a}} (t)}{\bm{a}\in X}$.

Now $M_X(\R)$ acts on $C(X)$ by
$(M f)(\bm{x})
=\sum_{\bm{y}\in X}M_{\bm{x},\bm{y}} f(\bm{y})$
for $M\in M_X(\R)$ and $f\in C(X)$.
Let $F_i$ be the projection matrix
from $C(X)$ onto $\Harm_i(X)$,
that is,
$F_i f $ is equal to $f$
if $f\in \Harm _i(X)$,
0 if $f\in \bigoplus^S_{j=0,j\neq i}\Harm _j(X)$.
Then $\{F_i\}^{S}_{i=0}$ satisfy
$\sum^{S}_{i=0}F_i=I$, $F_j\neq 0$ and $F_j^2=F_j$
for $j\in \{0,1,\ldots ,S\}$.
We call $\{F_i\}^{S}_{i=0}$
the \emph{projection matrices}
of $X$.

\begin{lem}
$F_0=\frac{1}{n}J$ and $F_1=G$.
\end{lem}
\begin{proof}
Since
$(\frac{1}{n}J\cdot \pmb{1})(\bm{x}) =\frac{1}{n}\sum_{\bm{x}\in X}1
=\pmb{1}(\bm{x})$
for every $\bm{x}\in X$
and every $f \in \Harm_0 (X)^{\bot}$ satisfies
$(\frac{1}{n}J f )(\bm{x})=\frac{1}{n}\sum_{\bm{x}\in X}1\cdot f(\bm{x})
=(\pmb{1},f)=0$,
it follows that $F_0=\frac{1}{n}J$.

By (TD\ref{def:poly_sp_2}),
we have
\begin{eqnarray*}
(G \zeta _{\bm{y}}(t))(\bm{x})
&=&\sum_{\bm{z}\in X} 
\frac{1}{n}(\bm{x}\cdot \bm{z})
(\bm{z}\cdot \bm{y})
\\
&=& \bm{y}\cdot \bm{x}\\
&=& \zeta _{\bm{y}}(t) (\bm{x}).
\end{eqnarray*}
Also,
every $f \in \Harm_1 (X)^{\bot}$
satisfies
$(G f)(\bm{x})=
\sum_{\bm{y}\in X}
\frac{1}{n}(\bm{x}\cdot \bm{y}) f(\bm{y})
=
(\zeta_{\bm{x}} (t), f)
=0$.
Therefore, it follows that
$G=F_1$.
\end{proof}

\begin{lem}
\label{lem:F_j G^i=0}
For $j>i$,
$F_j G^{\circ i}=0$.
\end{lem}

\begin{proof}
The $\bm{y}$-th column of $G^{\circ i}$
can be regarded
as the zonal polynomial $\zeta_{\bm{y}}(t^i)$ of $t^i$ at $\bm{y}$.
Then $\zeta_{\bm{y}}(t^i)\in \pol _i(X)=\bigoplus^i_{k=0}\Harm _k(X)$.
Hence we get $F_j G^{\circ i}=0$.
\end{proof}

Suppose $X$ is an $s$-distance set
with the normalized Gram matrix $G$.
For $\alpha \in A'(X)$,
denote $|\set{(\bm{x},\bm{y})\in X^2}
{\bm{x}\cdot \bm{y}=\alpha}|$ by $\kappa_\alpha$.
We put $Z^\ast(t)=\prod_{\alpha \in A'(X)}(t-\alpha)$.
We define an inner product on $\mathbb{R}[t]/(Z^\ast)$ by,
for $p,q\in \mathbb{R}[t]/(Z^\ast)$,
\begin{equation}
\label{eq:poly_inn}
\aprod{p}{q}=
\frac{1}{n^2}\sum_{\alpha \in A'(X)}\kappa_\alpha
p(\alpha)q(\alpha).
\end{equation}
The {\it predegree polynomials}
$q_0,q_1,\ldots ,q_s$
of $X$ are the polynomials satisfying
$\deg q_k =k$ and $\aprod{q_k}{q_h}=\delta _{k,h} q_k(m)$
for any $k,h\in \{0,1,\ldots ,s\}$.
As a sequence of orthogonal polynomials, the predegree polynomials
satisfy a three-term recurrence of the form
\begin{equation}
\label{eq:q_i recurrence}
t q_k =b^\ast_{k-1} q_{k-1}+a^\ast_{k} q_{k}+c^\ast_{k+1} q_{k+1}
\quad (0\le k\le s),
\end{equation}
where the constants $b^\ast_{k-1}$, $a^\ast_{k}$ and $c^\ast_{k+1}$
are the Fourier coefficients of $t q_k$ in terms of
$\{q_i\}^s_{i=0}$
respectively (and $b^\ast_{-1}=c^\ast_{s+1}=0$).
Moreover, the predegree polynomials satisfy
the following equation,
cf.~\cite{Fiol1997fla}:
\[
\sum^{s}_{i=0}q_i (t)=n \prod_{\alpha \in A(X)}
\frac{t-\alpha}{m-\alpha}.
\]
To see this,
just note that
$\aprod{n \prod_{\alpha \in A(X)}
\frac{t-\alpha}{m-\alpha}}{q_k}=\frac{1}{n}\kappa _m q_k(m)=q_k(m)$.
We put $H(t)=\sum^{s}_{i=0}q_i (t)$.
Then $\frac{1}{n} H((n G)^\circ)= I$.

\begin{lem}
\label{lem:Q-poly->dis}
Let $\mathfrak{X}=(X,\{R_i\}^{s}_{i=0})$ be a $Q$-polynomial scheme,
with $|X|=n$, with respect to $E_1$.
Then the degree $S$ of the image of
the spherical embedding of $\mathfrak{X}$
with respect to $E_1$
is $s$.
Moreover,
for each $i\in \{0,1,\ldots ,s\}$,
$F_i=
\frac{1}{n}q_i ((n G)^\circ)$.
\end{lem}
\begin{proof}
Let $\{E_i\}^s_{i=0}$ be the $Q$-polynomial ordering
of $\mathfrak{X}$,
and let $\{v^\ast_i\}^s_{i=0}$
be the polynomials
such that $E_i=\frac{1}{n}v^\ast _i((n E_1)^\circ)$
for each $i\in \{0,1,\ldots ,s\}$.
By Lemma~\ref{lem:Pol_span_easy},
we obtain
$\pol _i (X)=
\Span _\R \set{\zeta _a (p)}{a\in X,\ \deg p\le i}=
\sum^i_{j=0}(E_1)^{\circ j} C(X)$.
Moreover,
we obtain
$\sum^i_{j=0}(E_1)^{\circ j} C(X)=E_0 C(X)\perp E_1 C(X)\perp
\cdots \perp E_iC(X)$.
Hence $\Harm _i(X)=E_i C(X)$ and
$F_i=E_i$ for each $i\in \{0,1,\ldots ,s\}$,
and thus $S=s$.
Finally we prove
$q_i=v^\ast_i$
for $i\in \{0,1,\ldots ,s\}$.
For $i,j\in \{0,1,\ldots ,s\}$,
we get
\begin{eqnarray*}
\aprod{v^\ast_i}{v^\ast_j}
&=&
\frac{1}{n^2}\sum_{a,b\in X}
v^\ast_i(a\cdot b)
v^\ast_j(a\cdot b)\\
&=&
\sum_{a,b\in X}
(E_i)_{a,b}
(E_j)_{a,b}\\
&=&
\delta_{i,j}\tr E_i\\
&=&
\delta_{i,j}v^\ast_i(m),
\end{eqnarray*}
that is $\{v^\ast _i\}^s_{i=0}$ coincide with the predegree
polynomials $\{q_i\}^s_{i=0}$ of $X$.
Therefore we get the desired result.
\end{proof}

\begin{lem}
\label{lem:dis->Q-poly}
Suppose a spherical $2$-design $X$ is an $s$-distance set
with $S=s$.
If $\frac{1}{n}q_i((n G)^\circ)=F_i$
for $i\in \{0,1,\ldots ,s\}$,
then $X$ carries
a $Q$-polynomial association scheme.
\end{lem}
\begin{proof}
Since $n G=\sum_{\alpha\in A'(X)} \alpha A_\alpha$,
we get
\[
F_i=\frac{1}{n}q_i((n G)^\circ)
=\frac{1}{n}
\sum_{\alpha\in A'(X)} q_i(\alpha) A_\alpha
\in \mathfrak{A}:=\Span _\R \set{A_\alpha}{\alpha \in A'(X)}.
\]
Comparing dimensions,
we find $\mathfrak{A}=\Span _\R \{F_i\}^s_{i=0}$.
Since $F_0,F_1,\ldots ,F_s$ are orthogonal projections,
it follows that
$\mathfrak{A}$ is closed under the multiplication.
Hence we proved that $(X, \{R_\alpha\}_{\alpha \in A'(X)})$
is an association scheme,
where $R_\alpha$ is the relation whose adjacency matrix is $A_\alpha$. 

Now
$\{F_i\}^s_{i=0}$ are the primitive idempotents of $\mathfrak{A}$.
Moreover
$(X,\{R_\alpha\}_{\alpha \in A'(X)})$
is a $Q$-polynomial association scheme
because of
$F_i=\frac{1}{n}q_i((n G)^\circ)=\frac{1}{n}q_i((n F_1)^\circ)$.
\end{proof}
From Lemmas \ref{lem:Q-poly->dis}
and \ref{lem:dis->Q-poly},
we proved that
a spherical 2-design which is an $s$-distance set and $S=s$
has the structure of a $Q$-polynomial scheme
with respect to the idempotents
$\{F_i\}^s_{i=0}$
if and only if
$\frac{1}{n}q_i((n G)^\circ)=F_i$ for $i\in \{0,1,\ldots ,s\}$.

\section{An excess theorem for spherical 2-designs}
\label{sec:excess}
Recall that the spectral excess theorem
implies that,
for a graph $\Gamma=(X,E)$,
the mean value of excesses $\{|\Gamma _d(x)|\}_{x\in X}$
is bounded above by using the predistance polynomial $p_d$
of degree $d$
and equality holds if and only if
$\Gamma$ is distance-regular.
In this section
we give an excess theorem for a spherical
2-design $X$ which is an $s$-distance set and $S=s$.
For $\bm{x}\in X$, put $m_s(\bm{x})=n(F_s)_{\bm{x},\bm{x}}$,
and $m_s(\bm{x})$ is called the \emph{excess} of $\bm{x}$
in terms of $X$.
The following theorem is the main theorem in this paper.

\begin{thm}
\label{thm:excess theorem main}
Suppose a spherical
$2$-design $X$ is an $s$-distance set
with $S=s$.
Then the inequality
\[
\tr F_s=
\frac{1}{n}\sum_{\bm{x}\in X}m_s(\bm{x})\le q_s(m)
\]
holds
and equality is attained if and only if
$X$ has the structure of a $Q$-polynomial association scheme
with respect to the idempotents $\{F_i\}^{s}_{i=0}$.
\end{thm}

In order to prove Theorem~\ref{thm:excess theorem main},
we give some notation and a lemma.
The set $\mathfrak{A}^\ast =\mathfrak{A}^\ast (X)=
\set{\frac{1}{n}p((n G)^\circ)}{p\in \mathbb{R}[t]}$
is a vector space of dimension $s+1$
and $\{J, G, \ldots ,G^{\circ s}\}$ is a basis of $\mathfrak{A}^\ast$.
Let $\mathfrak{D}^\ast=\mathfrak{D}^\ast (X)$ be the linear span
of the set $\{F_i\}^s_{i=0}$.
In our context, we will work with the vector space
$\mathfrak{T}^\ast =\mathfrak{A}^\ast +\mathfrak{D}^\ast$.
Note that $J$, $G$ and $I$ are matrices in
$\mathfrak{A}^\ast \cap\mathfrak{D}^\ast$
since $I=\frac{1}{n} H((n G)^\circ)\in \mathfrak{A}^\ast$.
Moreover, we have:
\begin{equation}
\label{eq:F q_i(G) I}
F_0+F_1+\cdots +F_s=I =
\frac{1}{n}(q_0((n G)^\circ)+q_1((n G)^\circ)+\cdots
+q_s((n G)^\circ)).
\end{equation}

\begin{lem}
\label{lem:F_s<->Q-poly}
Suppose a spherical
$2$-design $X$ is an $s$-distance set
with projection matrices $\{F_i\}^s_{i=0}$.
Then $X$ has the structure of a $Q$-polynomial association scheme
with respect to the idempotents $\{F_i\}^{s}_{i=0}$
if and only if $F_s =\frac{1}{n}q_s((n G)^\circ)$.
\end{lem}
\begin{proof}
By Lemma \ref{lem:Q-poly->dis},
if $X$ has the structure of a $Q$-polynomial association scheme
with respect to the idempotents $\{F_i\}^{s}_{i=0}$,
then $F_s =\frac{1}{n}q_s((n G)^\circ)$.
We prove that if $F_s =\frac{1}{n}q_s((n G)^\circ)$, then
$X$ has the structure of a $Q$-polynomial association scheme
with respect to the idempotents $\{F_i\}^{s}_{i=0}$.
By Lemma~\ref{lem:dis->Q-poly},
it suffices to show
$F_k=\frac{1}{n}q_k((n G)^\circ)$ for each
$k\in \{2,3,\ldots ,s-1\}$.
First we check $F_{s-1}=\frac{1}{n}q_{s-1}((n G)^\circ)$.
From (\ref{eq:F q_i(G) I}) and
the recurrence (\ref{eq:q_i recurrence})
for the predegree polynomials,
we obtain the following two equalities:
\begin{equation}
\label{eq:F_s-1 q_s-1}
F_0+F_1+\cdots +F_{s-1}
=
\frac{1}{n}q_0((n G)^\circ)+
\frac{1}{n}q_1((n G)^\circ)+
\cdots+
\frac{1}{n}q_{s-1}((n G)^\circ),
\end{equation}
\begin{equation}
\label{eq:recc2}
(n G)\circ F_s =b^\ast_{s-1} \frac{1}{n}q_{s-1} ((n G)^\circ)
+ a^\ast_s F_s.
\end{equation}
Then for  $\varphi\in \Harm _s (X)$,
we have
\begin{eqnarray*}
\left(\frac{1}{n}q_{s-1} ((n G)^\circ) \varphi\right)(\bm{x})
&=&
\frac{1}{n}\sum_{\bm{y}\in X}\zeta _{\bm{x}}
(q_{s-1})(\bm{y})\varphi(\bm{y})\\
&=&(\zeta _{\bm{x}}(q_{s-1}), \varphi)=0
\end{eqnarray*}
for every $\bm{x}\in X$,
since $\zeta _{\bm{x}}(q_{s-1})$ is in $\pol_{s-1}(X)$.
For $\psi\in \bigoplus^{s-2}_{i=0}\Harm _i (X)$,
we have
\begin{eqnarray*}
\left((n G)\circ F_s \psi\right)(\bm{x})
&=&
\sum_{\bm{y}\in X}(\bm{x}\cdot \bm{y}) (F_s)_{\bm{x},\bm{y}} \psi(\bm{y})\\
&=&
\sum_{\bm{y}\in X} (F_s)_{\bm{x},\bm{y}} \zeta_{\bm{x}}(t)
(\bm{y})\psi(\bm{y})\\
&=&
(F_s \cdot \zeta_{\bm{x}}(t) \psi)(\bm{x})=0
\end{eqnarray*}
for every $\bm{x}\in X$,
since $\zeta_{\bm{x}}(t) \psi$ is in $\pol_{s-1}(X)$.
Thus, the equality (\ref{eq:recc2}) implies that
\[
\left(
\frac{1}{n}q_{s-1} ((n G)^\circ) \psi
\right)(\bm{x})=0.
\]
Hence, for every $f\in C(X)$, we have
$\frac{1}{n}q_{s-1} ((n G)^\circ) f \in \Harm_{s-1} (X)$.
Multiplying both sides of (\ref{eq:F_s-1 q_s-1}) from the left
by $F_{s-1}$,
we have
\begin{equation}
\label{eq:F_s-1}
F_{s-1}=F_{s-1} \frac{1}{n}q_{s-1} ((n G)^\circ),
\end{equation}
by Lemma \ref{lem:F_j G^i=0}.
For $\chi \in \Harm_{s-1} (X)$,
we put $\tilde{\chi }=\frac{1}{n}q_{s-1} ((n G)^\circ) \chi$.
Then, from (\ref{eq:F_s-1}), we have
\begin{eqnarray*}
\tilde{\chi}
&=&
F_{s-1} \tilde{\chi }
=
F_{s-1} \frac{1}{n}q_{s-1} ((n G)^\circ) \chi \\
&=&
F_{s-1} \chi =\chi .
\end{eqnarray*}
Thus, it follows that $\frac{1}{n}q_{s-1} ((n G)^\circ) =F_{s-1}$.

Let $2\le k\le s-2$ and
suppose now that $\frac{1}{n}q_{i} ((n G)^\circ) =F_{i}$ for
$s\ge i\ge k+1$.
Then we have the following two equalities:
\begin{equation}
\label{eq:F_k q_k}
F_0+F_1+\cdots +F_{k}
=
\frac{1}{n}q_0((n G)^\circ)+
\frac{1}{n}q_1((n G)^\circ)+
\cdots+
\frac{1}{n}q_{k}((n G)^\circ),
\end{equation}
\begin{equation}
\label{eq:recc k}
(n G)\circ F_{k+1} 
=b^\ast _{k} \frac{1}{n}q_{k} ((n G)^\circ)
+a^\ast _{k+1} F_{k+1}+ c^\ast _{k+2} F_{k+2}.
\end{equation}
As before, from $\deg q_k=k$,
we infer that
$\frac{1}{n}q_{k} ((n G)^\circ) \varphi=0$
for $\varphi\in \bigoplus^{s}_{i=k+1}\Harm _i(X)$.
By the same reasoning as above, (\ref{eq:recc k}) yields
$\frac{1}{n}q_{k} ((n G)^\circ) \psi=0$
for $\psi\in \bigoplus^{k-1}_{i=0}\Harm _i(X)$.
Hence, for any $f\in C(X)$, we have $\frac{1}{n}q_{k} ((n G)^\circ) f
\in \Harm _k(X)$.
For $\chi \in \Harm_k (X)$,
we can similarly prove $\frac{1}{n}q_{k} ((n G)^\circ) \chi =\chi$.
Then we have that $\frac{1}{n}q_{k} ((n G)^\circ) =F_{k}$
which, by induction, proves the result.
\end{proof}

We define an inner product in $M_X(\R)$ by,
for $R,S\in M_X(\R)$,
\begin{equation}
\label{eq:mat_inn}
\mprod{R}{S}=\tr ({}^t R S).
\end{equation}
Observe that, for $p,q \in \R[t]$, we have
\begin{equation}
\mprod{\frac{1}{n}p((n G)^\circ)}{\frac{1}{n}q((n G)^\circ)}
=\frac{1}{n^2}\sum_{\alpha \in A'(X)}
\kappa_\alpha p(\alpha)q(\alpha)
=\aprod{p}{q}.
\end{equation}
Consider
the Euclidean space $\mathfrak{T}^\ast$ with the inner product
(\ref{eq:mat_inn})
and the orthogonal projection
$\mathfrak{T}^\ast\rightarrow \mathfrak{A}^\ast$ denoted by
$S\mapsto \tilde{S}$.
Using in $\mathfrak{A}^\ast$ the orthogonal basis
$\{\frac{1}{n}q_i((n G)^\circ)\}^{s}_{i=0}$
of the predegree polynomials of $X$, this projection can be
expressed as
\begin{equation}
\tilde{S} =\sum^{s}_{i=0}
\frac{\mprod{S}{q_i((n G)^\circ )}}{n^2 q_i(m)}
q_i((n G)^\circ ).
\end{equation}

We put $\mu =\frac{1}{n}\sum_{\bm{x}\in X}m_s(\bm{x})$, for short.
\begin{proof}[Proof of Theorem~\ref{thm:excess theorem main}]
By Lemma~\ref{lem:F_j G^i=0} and since $\mprod{F_s}{F_s}
=\mprod{F_s}{I}
=\tr(F_s)
=\mu$,
the projection of $F_s$ onto $\mathfrak{A}^\ast$ is
\begin{eqnarray}
\label{eq:proj_F_s}
\tilde{F_s}
&=&
\sum^{s}_{i=0}
\frac{\mprod{F_s}{q_i((n G)^\circ )}}{n^2 q_i(m)}
q_i((n G)^\circ )
=\frac{\mprod{F_s}{q_s((n G)^\circ )}}{n^2 q_s(m)}q_s((n G)^\circ )
\\
&=&
\notag
\frac{\mprod{F_s}{H((n G)^\circ )}}{n^2 q_s(m)}q_s((n G)^\circ )
=\frac{\mprod{F_s}{I}}{n q_s(m)}q_s((n G)^\circ )\\
&=&
\notag
\frac{\mu}{n q_s(m)}q_s((n G)^\circ ).
\end{eqnarray}
Hence we get
\[
0\le 
\mprod{F_s}{F_s}-\mprod{\tilde{F_s}}{\tilde{F_s}}
=
\mu- \frac{\mu ^2}{q_s(m)}
=
\mu \left(
1-\frac{\mu }{q_s(m)}
\right).
\]
Since $\mu =\tr (F_s)=\rank (F_s)>0$,
this implies the desired inequality.
Moreover the equality is attained if and only
$F_s=\tilde{F_s}=\frac{1}{n}q_s((n G)^\circ)$.
By Lemma \ref{lem:F_s<->Q-poly} the desired result follows.
\end{proof}

\noindent
\textbf{Acknowledgments.} 
The author would like to thank 
Eiichi Bannai,
Edwin van Dam, 
Tatsuro Ito, 
Jack Koolen,
William J. Martin, 
Akihiro Munemasa
and 
Hajime Tanaka
for useful discussions and comments.
Part of this work was carried out while the author was visiting
the Mathematisches Forschungsinstitut Oberwolfach.
The institute kindly offered the
stay while his University was affected by the 2011 Tohoku earthquake.
The author is grateful to the MFO and its staffs for the offer,
financial support and warm hospitality.
The author also would like to thank the Leibniz Association for travel support.
The author is supported by the fellowship of the Japan Society for
the Promotion of Science.

\end{document}